\documentclass{amsproc}
\usepackage{amsfonts}
\usepackage[page]{appendix}

\theoremstyle{plain}

\newtheorem{condition}{Condition}

\newtheorem{corollary}{Corollary}

\newtheorem{definition}{Definition}
\newtheorem{example}{Example}

\newtheorem{proposition}{Proposition}
\newtheorem{remark}{Remark}

\newtheorem{theorem}{Theorem}
\numberwithin{equation}{section}
\input{tcilatex}

\begin{document}
\title[Positivity]{On positivity of orthogonal series and its applications
in probability}
\author{Pawe\l\ J. Szab\l owski}
\address{Emeritus in Department of Mathematics and Information Sciences,
Warsaw University of Technology ul Koszykowa 75, 00-662 Warsaw, Poland}
\email{pawel.szablowski@gmail.com}
\urladdr{}

\begin{abstract}
We give necessary and sufficient conditions for an orthogonal series to
converge in the mean-squares to a nonnegative function. We present many
examples and applications, in analysis and probability. In particular, we
give necessary and sufficient conditions for a Lancaster-type of expansion $%
\sum_{n\geq 0}c_{n}\alpha _{n}(x)\beta _{n}(y)$ with two sets of orthogonal
polynomials $\left\{ \alpha _{n}\right\} $ and $\left\{ \beta _{n}\right\} $
to converge in means-squares to a nonnegative bivariate function. In
particular, we study the properties of the set $C(\alpha ,\beta )$ of the
sequences $\left\{ c_{n}\right\} ,$ for which the above-mentioned series
converge to a nonnegative function and give conditions for the membership to
it. Further, we show that the class of bivariate distributions for which a
Lancaster type expansion can be found, is the same as the class of
distributions having all conditional moments in the form of polynomials in
the conditioning random variable.
\end{abstract}

\maketitle
\thanks{The author is very grateful to the unknown referee for suggestions
improving the readability of the paper and also for pointing out numerous
misprints.}

\section{Introduction}

\subsection{Notation, terminology and basic settings}

First, let us fix notation that mostly comes from the measure theory. All
signed measures considered in the paper will be $\sigma -$finite,
consequently, the Radon-Nikodym theorem can be applied. If $\chi $ is a
signed measure and $\chi \allowbreak =\allowbreak \chi ^{+}\allowbreak
-\allowbreak \chi ^{-}$ is its Hahn-Jordan decomposition then $|\chi
|\allowbreak =\allowbreak \chi ^{+}\allowbreak +\allowbreak \chi ^{-}$ is a
measure. Obviously, a signed measure $\chi $ is a measure if $\chi
^{-}\allowbreak =\allowbreak 0$. We will use the notation $\int f(x)d\mu (x)$
interchangeably with $\int f(x)\mu (dx)$ or even $\int fd\mu $ if the set of
integration is evident, to denote integral with respect to the (possibly
signed) measure $\mu $. Sometimes $d\mu (.)$ will denote measure $\mu $
itself.

Let $L_{2}(\limfunc{supp}(\mu ),\mu )$ denote the set of all functions $f:%
\mathbb{R}^{m}\longrightarrow \mathbb{R}$ that are square-integrable with
respect to the measure $\left\vert \mu \right\vert $. Let us also agree that
since all functions from the set $L_{2}(\limfunc{supp}(\mu ),\mu )$ are
defined only on $\limfunc{supp}(\mu )$ ($\mu \allowbreak -\allowbreak a.s.$
in fact), hence we will use notation $L_{2}(\mu )$ instead $L_{2}(\limfunc{%
supp}(\mu ),\mu )$.

In the sequel, we will be interested only in signed measures that have
one-dimensional marginal measures that are identified by their moments (for
the definition and basic properties see the Appendix below). Following \cite%
{Chih79} or \cite{Sim98} this is assured for those one-dimensional measures $%
\mu $ that they satisfy the so-called Cramer's condition\footnote{%
The names of this condition as well as Hardy's condition, below, were
recalled by Prof. Jordan Stoyanov in a private letter.}, that is that there
exists $\delta >0$ such that 
\begin{equation}
\int \exp (\delta \left\vert x\right\vert )d\left\vert \mu \right\vert
(x)<\infty .  \label{IM}
\end{equation}%
In fact condition (\ref{IM}) can have a weaker form (i.e. the so-called
Hardy's condition if the measure $\mu $ has support contained in $\left\{
x:x\geq 0\right\} $. But we will not go into these details.

Let us denote by $Cra$ the set of all signed measures $\chi $ on $\mathbb{R}$
such that satisfy the condition (\ref{IM}) for some positive number $\delta $%
. Notice that $Cra$ contains all measures with bounded supports.

Further, let us introduce the following set of signed measures $AC2(\mu ),$
generated by a measure $\mu :$%
\begin{equation*}
AC2(\mu )=\left\{ fd\mu :f\in L_{2}(\mu )\right\} .
\end{equation*}

In other words, the set $AC2(\mu )$ contains all signed measures $\nu $ that
are absolutely continuous with respect to $\mu $ with their Radon-Nikodym
derivative $\frac{d\nu }{d\mu }$ (i.e. function $f$) being square integrable
with respect to the measure $\mu .$ Note, that in the definition of the set $%
AC2(\mu )$, $\mu $ can be a multidimensional $\sigma -$finite measure.

We have the following simple observation:

\begin{proposition}
If a one-dimensional measure  $\chi $ belongs to the set $Cra,$ then $%
AC2\left( \chi \right) \subset Cra$.
\end{proposition}

\begin{proof}
Let $f\in L^{2}(\chi )$, then by by the Cauchy-Schwarz inequality we have: 
\begin{equation*}
(\int_{\limfunc{supp}\mu }\left\vert f(x)\right\vert \exp (\delta
|x|/2)d\left\vert \chi \right\vert (x))^{2}\leq \int_{\limfunc{supp}\mu
}\left\vert f\right\vert ^{2}d\left\vert \chi \right\vert \int_{\limfunc{supp%
}\mu }\exp (\delta \left\vert x\right\vert )d\left\vert \chi \right\vert
(x)<\infty .
\end{equation*}%
Consequently, $fd\chi $ satisfies condition (\ref{IM}), that is, it belongs
to the set $Cra$.
\end{proof}

In other words, if a signed measure $\mu $ is identifiable by moments, then
every element of $AC2\left( \mu \right) $ is identifiable by moments.

Let $\mu $ be a measure from the set $Cra,$ by $AC2^{+}(\mu )$ let us denote
the subset of $AC2(\mu )$ that contains only measures, i.e. $%
\int_{A}f(x)d\mu (x)\geq 0$ for any $\mu -$ measurable set $A\subset \mathbb{%
R}$. Notice that $AC2^{+}(\mu )$ is in fact a closed cone in $AC2(\mu )$.

Now if $\mu $ is a measure, $\left\{ p_{n}\right\} _{n\geq 0}$ the set of
polynomials that are orthogonal with respect to the measure $\mu .$ Let us
also define numbers $\hat{p}_{n}$ by the following orthogonality relationship%
\begin{equation}
\hat{p}_{n}\allowbreak \delta _{mn}=\allowbreak \int_{\limfunc{supp}\mu
}p_{n}(x)p_{m}(x)d\mu (x),  \label{sqp}
\end{equation}
with $\delta _{nm}$ denoting traditionally, the Kronecker's delta. Let us
agree that in the sequel the "hat" over the symbol of a polynomial will
denote the positive number defined by (\ref{sqp}). Further, let $\left\{
c_{n}\right\} _{n\geq 0}$ be the sequence of reals. The infinite series of
the form 
\begin{equation}
\sum_{n\geq 0}c_{n}p_{n}(x),  \label{Ors}
\end{equation}%
is called an orthogonal series. It is known (see, e.g. \cite{Alexits61}),
that if the following condition 
\begin{equation}
\sum_{n\geq 0}c_{n}^{2}\hat{p}_{n}<\infty   \label{ms}
\end{equation}%
is satisfied, then the series (\ref{Ors}) converges in $L_{2}(\mu )$.

\begin{remark}
Let us recall, that basically, all series considered will converge in the
mean-square sense for some specified measure. However, let us recall that
due to the Rademacher-Men'shov theorem (see, e.g. \cite{Alexits61}),
assuming sometimes only little stronger conditions, one can obtain the
convergence almost everywhere with respect to the specified measure. More
precisely, the Rademacher-Meshov theorem states that if the following
condition is satisfied%
\begin{equation}
\sum_{n\geq 0}c_{n}^{2}\hat{p}_{n}\log ^{2}(n+1)<\infty ,  \label{RM}
\end{equation}%
then the series (\ref{Ors}) converges in the means-squares and almost
everywhere $\func{mod}\mu .$
\end{remark}

\subsection{The problem}

The main idea of the paper is to present necessary and sufficient conditions
for the positivity of the sum of the orthogonal series (\ref{Ors}) for
almost all $x$ belonging to the closed subset $\mathcal{M}$ of the support
of the positive measure $\mu $ with respect to which polynomials $\left\{
p_{n}\right\} $ are orthogonal. The necessary part of the theorem has been
presented in 2011 in \cite{Szablowski2010(1)}. Later over the years, slight
generalizations of the original formulation and many examples were presented
in \cite{SzabChol}, \cite{SzablKer}.

However, only recently I have realized, that the necessary conditions for
the coefficients $c_{n}$ to assure positivity of (\ref{Ors}), are also
sufficient.

The paper is organized as follows. In the next Section \ref{gl}, we present
our main result together with its simple proof. We also quote papers where
many examples illustrating the assertions of the theorem are presented. The
last Section \ref{prob}, presents applications of our result to probability
theory in particular to the so-called Lancaster expansions. There is also an
appendix in which we recall basic facts about the moments and the moments'
problem.

\section{General results\label{gl}}

Our main result is the following :

\begin{theorem}
\label{main}Let a measure $\mu \in Cra$, $\left\{ p_{n}\right\} $ be the
sequence of polynomials orthogonal with respect to the measure $\mu $. Let
us consider the orthogonal series (\ref{Ors}) and by $f$ let us denote the
mean-square sum of it. Let $\mathcal{M}$ be some closed subset of the $%
\limfunc{supp}(\mu ).$

The following two conditions are equivalent:

a) $f\left( x\right) \geq 0$ $\mu $ - a.s. on $\mathcal{M}$,

b) There exists $\nu \in AC2^{+}(\mu )$, with $\limfunc{supp}(\nu
)\allowbreak =\allowbreak \mathcal{M}$ such that $f(x)\allowbreak
=\allowbreak \frac{d\nu }{d\mu }(x)$.

If one of the conditions a) and b) is satisfied, then the coefficients $%
c_{n} $ are given by the following formula: 
\begin{equation}
c_{n}\allowbreak =\allowbreak (\int_{\mathcal{M}}p_{n}(x)d\nu (x))/\hat{p}%
_{n}.  \label{cn}
\end{equation}
\end{theorem}

\begin{remark}
Notice, that, if additionally, condition (\ref{RM}) is satisfied by the
coefficients $\left\{ c_{n}\right\} $, then the convergence to $f$ is not
only in mean-square but also almost surely for almost all $x\in \mathcal{M}$ 
$\func{mod}\mu $.
\end{remark}

\begin{remark}
Notice, that, if $\left\{ r_{n}\right\} $ denotes the sequence of
polynomials orthogonal with respect to the measure $\nu $, then the number $%
\int_{\limfunc{supp}\mu }p_{n}(x)d\nu (x)$ is equal to the free coefficient
in the connection coefficient expansion of $p_{n}(x)$ in terms of $\left\{
r_{n}\right\} .$ More precisely considering connection coefficient expansions%
\begin{equation}
p_{n}(x)=\sum_{j=0}^{n}\gamma _{n,j}r_{j}(x),  \label{pnar}
\end{equation}%
we have 
\begin{equation*}
\int_{\limfunc{supp}\mu }p_{n}(x)d\nu (x)=\gamma _{n,0}.
\end{equation*}%
Consequently, the assertion of the Theorem \ref{main} can be rephrased in
the following way.

An orthogonal series (\ref{Ors}) with coefficients satisfying condition (\ref%
{ms}), is nonnegative for almost all (mod $\mu )$ $x\in \limfunc{supp}\mu $
if and only if another sequence $\left\{ r_{n}\right\} $ of orthogonal
polynomials can be found such that considering connection coefficient
expansion of $p_{n}(x)$ in terms of $\left\{ r_{n}\right\} $ given by (\ref%
{pnar}) we have:%
\begin{equation}
c_{n}=\gamma _{n,0}/\hat{p}_{n}.  \label{cgp}
\end{equation}
\end{remark}

\begin{proof}[Proof of Theorem \protect\ref{main}]
b) $\Rightarrow $ a). First, let us assume that the coefficients $c_{n}$ are
given by (\ref{cn}) and let us denote by $f(x)$ the sum of (\ref{Ors}). By
assumptions, we know that it exists and it is square-integrable with respect
to $\mu .$ We have 
\begin{equation*}
\int_{\mathcal{M}}p_{n}(x)f\left( x\right) d\mu (x)\allowbreak =\allowbreak
c_{n}\allowbreak \hat{p}_{n}=\int_{\mathcal{M}}p_{n}(x)d\nu (x).\allowbreak 
\end{equation*}%
Knowing numbers $c_{n}\hat{p}_{n}$, $n\allowbreak =\allowbreak 0,1,2\ldots $
and the form of polynomials $\left\{ p_{n}(x)\right\} .$ we can find numbers 
$\left\{ \int_{\mathcal{M}}x^{n}d\nu (x)\right\} _{n\geq 0}$ and $\left\{
\int_{\mathcal{M}}x^{n}f(x)d\mu (x)\right\} $. We see that they are
identical and the two measures are, by assumption, identifiable by moments
so the two measures must be identical i.e. 
\begin{equation*}
f\left( x\right) d\mu (x)=d\nu (x).
\end{equation*}%
But $\nu $ was chosen to be nonnegative. So $f(x)\geq 0$ on the $\mathcal{M}$
mod $\mu $. Besides, we see that 
\begin{equation*}
f(x)=\frac{d\nu }{d\mu }(x).
\end{equation*}

a) $\Rightarrow $ b). Now, let us assume, that we want to find an expansion
of the Radon-Nikodym derivative of two nonnegative measures $\nu <<\mu $
that is additionally square-integrable (mod $\mu )$ in an infinite
orthogonal series. That is, we are looking for the coefficients of the
expansion of the form of (\ref{Ors}). Then, following our assumptions, we
have%
\begin{eqnarray*}
d\nu (x) &=&f(x)d\mu (x), \\
p_{n}(x) &=&\sum_{j=0}^{n}\gamma _{n,j}r_{j}(x),
\end{eqnarray*}%
where $\left\{ r_{n}(x)\right\} $ are polynomials orthogonal with respect to 
$\nu $. These polynomials exist since for every positive measure satisfying
condition (\ref{IM}) one can define such polynomials. Naturally, having two
sets of orthogonal polynomials one has a set of connection coefficients
between them. Since $f$ is square-integrable with respect to $d\mu (x)$ we
know that the coefficients $c_{n}$ are defined uniquely. Besides we have%
\begin{equation*}
\int_{\mathcal{M}}f(x)p_{n}(x)d\mu \left( x\right) =c_{n}\hat{p}_{n}=\int_{%
\mathcal{M}}p_{n}(x)d\nu (x)\allowbreak =\allowbreak \gamma _{n,0}.
\end{equation*}
\end{proof}

There are numerous examples of expansions of the type (\ref{Ors}). They
appeared over the years in \cite{Szablowski2010(1)} (Section 5 concerning
mostly polynomials from the so-called Askey-Wilson scheme) or recently in 
\cite{Szab2020}, as well as in \cite{SzablKer}, \cite{Szab-bAW}, \cite%
{Szab13} ( concerning. among others, Charlier (3.7) or Jacobi (3.6)
polynomials).

\begin{remark}
\label{moment}Notice also that coefficient $\gamma _{n,0}$ is equal to 
\begin{equation}
\gamma _{n,0}=\sum_{j=0}^{n}\pi _{n,j}m_{j},  \label{gam1}
\end{equation}%
where the coefficients $\left\{ \pi _{n.j}\right\} $ are defined by the
expansion 
\begin{equation*}
p_{n}(x)=\sum_{j=0}^{n}\pi _{n,j}x^{j},
\end{equation*}%
while the numbers $\left\{ m_{j}\right\} $ form a moment sequence of some
distribution absolutely continuous with respect to the measure $\mu .$ This
observation can be derived directly from (\ref{cn}) or from the formula
given by Lemma 1 of \cite{SzabChol}).

This observation leads also to the following method of checking if a given
sequence $\left\{ c_{n}\right\} $ applied in the series (\ref{Ors}) can
result in the series' positive-sum. Namely, considering formulae (\ref{gam1}%
) and (\ref{cgp}) we can find a sequence $\left\{ m_{n}\right\} $ by
recursively solving a sequence of equations:%
\begin{equation*}
m_{n}=\frac{1}{\pi _{n,n}}(\hat{p}_{n}c_{n}-\sum_{j=0}^{n-1}\pi _{n,j}m_{j}),
\end{equation*}%
for $n\geq 0$. Since $\left\{ m_{n}\right\} $ has to be a moment sequence,
we can apply one of the known criteria some of which are presented in the
Appendix.
\end{remark}

\begin{remark}
Continuing the previous remark, the assertion of the theorem (in case when $%
\mathcal{M\allowbreak =}\limfunc{supp}\mu $ can be expressed in the
following way.

There exists a linear map: $K:$ $L^{2}(\mu )\longrightarrow L^{2}(\mu )$
that can be symbolically expressed by the following formula:%
\begin{equation*}
K(f)(x)=\allowbreak \int_{\limfunc{supp}\mu }(\sum_{i\geq 0}p_{i}(x)p_{i}(y)/%
\hat{p}_{n})f(y)d\mu (y),
\end{equation*}%
that maps every function $f\in L^{2}(\mu )$ on itself, since, as it is
easily seen, we have: 
\begin{equation*}
\int_{\limfunc{supp}\mu }(\sum_{i\geq 0}p_{i}(x)p_{i}(y)/\hat{p}%
_{n})(\sum_{n\geq 0}c_{n}p_{n}(y))d\mu (y)=\allowbreak \sum_{n\geq
0}c_{n}p_{n}(x).
\end{equation*}
\end{remark}

\section{Probabilistic aspects\label{prob}}

In this section to avoid confusion, we will assume that all considered
measures will be probabilistic that is they will integrate up to $1.$
Further, we will consider bivariate distributions $dF(x,y)$ with marginal
distributions $d\mu (x)$ and $d\nu (y)$ (i.e. $d\mu (x)\allowbreak
=\allowbreak \int F(x,dy)$ and similarly for $d\nu $). Naturally, we will
assume, that both marginal measures belong to the set $Cra$ in order to be
identified by their moments. Moreover, we will consider only such bivariate
distributions $F$ satisfying the following condition: 
\begin{equation}
\int_{\limfunc{supp}\mu }\int_{\limfunc{supp}\nu }(\frac{\partial ^{2}F}{%
\partial \mu \partial \nu }(x,y))^{2}d\mu (x)d\nu (y)<\infty ,  \label{L2}
\end{equation}%
where $\frac{\partial ^{2}F}{\partial \mu \partial \nu }(x,y)$, denotes
Radon-Nikodym derivative of the measure $F$ with respect to the product
measure $\mu \times v.$ That is, in other words, that $dF\in AC2(d\mu \times
d\nu ),$ where $d\mu \times d\nu $ denotes the product measure of $d\mu $
and $d\nu .$

Let us denote by $\left\{ \alpha _{n}(x)\right\} $ and $\left\{ \beta
_{n}(y)\right\} $ two sets of polynomials orthogonal with respect to the
measures respectively $d\mu (x)$ and $d\nu (y).$ Now, for all distributions
satisfying (\ref{L2}) the following expansion is valid:%
\begin{equation}
dF(x,y)=d\mu (x)d\nu (y)\sum_{i,j=0}\lambda _{i,j}\alpha _{i}(x)\beta
_{j}(y),  \label{exL2}
\end{equation}%
with $\sum_{i,j\geq 0}\lambda _{i,j}^{2}\hat{\alpha}_{i}\hat{\beta}%
_{j}<\infty .$ Conditional distributions $\zeta (dx|y)$ and $\xi (dy|x)$ are
defined respectively, for almost all $y$ ($\func{mod}v$) and almost all $x$ (%
$\func{mod}\mu )$ by the following relationships:%
\begin{equation}
F(dx,dy)=\zeta (dx|y)v(dy)=\xi (dy|x)\mu (dx).  \label{pr1}
\end{equation}%
One shows that both these distributions do exist and are defined uniquely
respectively $\func{mod}v$ and $\func{mod}\mu $.

Notice that making use of the definition of marginal distribution and the
orthogonal polynomials and changing, if necessary, the order of integration
that we have:%
\begin{equation*}
\int_{\limfunc{supp}\mu }\int_{\limfunc{supp}v}\beta _{n}(y)F(dx,dy)=\int_{%
\limfunc{supp}v}\beta _{n}(x)v(dx)=0,
\end{equation*}%
for all $n\geq 1$ and likewise for polynomials $\left\{ \alpha _{n}\right\} $%
.

Now applying the above-mentioned definitions and properties to the expansion
(\ref{exL2}) we deduce that 
\begin{equation}
\int \int_{\limfunc{supp}\mu \times \limfunc{supp}v}dF(x,y)\allowbreak
=\allowbreak \lambda _{0,0}\allowbreak =\allowbreak 1,~~\forall n\geq
1:\lambda _{0,n}=\lambda _{n,0}=0,  \label{pr3}
\end{equation}%
and also, that:%
\begin{equation}
\zeta (dx|y)=(\sum_{i,j=0}\lambda _{ij}\alpha _{i}(x)\beta _{j}(y))\mu
(dx),~\xi (dy|x)=(\sum_{i,j=0}\lambda _{ij}\alpha _{i}(x)\beta _{j}(y))v(dy).
\label{cond}
\end{equation}

We can now rephrase the above-mentioned Theorem \ref{main} in the form that
is important for the probabilists.

\begin{theorem}
\label{EXPA}Let $\mu \in Cra$, and let $\left\{ \alpha _{n}\right\} $ be a
set of polynomials orthogonal with respect to $\mu $. Then, the orthogonal
series:%
\begin{equation}
g(x,y)=\sum_{n\geq 0}h_{n}(y)\alpha _{n}(x)/\hat{\alpha}_{n},  \label{bv}
\end{equation}%
where, as above, $\hat{\alpha}_{n}$ is defined by (\ref{sqp}) and such that 
\begin{equation}
\sum_{m\geq 0}\left\vert h_{n}(y)\right\vert ^{2}/\hat{\alpha}_{n}<\infty ,
\label{h2}
\end{equation}%
for all $y$ belonging to some closed set $\limfunc{supp}\nu $, converges in
mean square (mod $\mu )$ to a nonnegative function iff there exists a family
of probability measures $\zeta (.|y)$ indexed by $y,$ such that for all $%
y\in \limfunc{supp}\nu ,$ $\zeta (.|y)<<\mu $ and $\forall n\geq 0:$ 
\begin{equation*}
h_{n}(y)\allowbreak =\allowbreak \int_{\limfunc{supp}\mu }\alpha
_{n}(x)\zeta (dx,y).
\end{equation*}%
Moreover 
\begin{equation*}
d\zeta (x|y)=g(x,y)d\mu (x).
\end{equation*}%
If additionally there exists a probability measure $\nu $ such that for $%
\forall n\geq 0:$ 
\begin{equation*}
\int_{\limfunc{supp}\nu }h_{n}(y)d\nu (y)\allowbreak =\allowbreak \delta
_{n,0},
\end{equation*}%
then one can define a bivariate measure $F<<\mu \times \nu $ by the formula%
\begin{equation}
F(dx,dy)=d\zeta (x|y)d\nu (y)=g(x,y)d\mu (x)d\nu (y)  \label{suma}
\end{equation}%
and for which 
\begin{equation*}
\mu (A)=\int_{\limfunc{supp}\nu }\zeta (A|y)d\nu (y),
\end{equation*}%
for all Borel subsets $A$ of $\limfunc{supp}(\mu )$ almost everywhere.
\end{theorem}

\begin{proof}
Suppose, that $F$ satisfies (\ref{L2}). Let $\mu $ and $\nu $ denote its
marginal measures and let the sets $\left\{ \alpha _{n}\right\} $ and $%
\left\{ \beta _{n}\right\} $ denote sets of polynomials orthogonal with
respect to measures $\mu $ and $\nu $ respectively. Let the conditional
distributions be defined by (\ref{cond}). Notice that by (\ref{pr3}) and (%
\ref{cond}) we have $\int_{\limfunc{supp}\mu }\zeta (dx,y)=1$ and similarly
for the $\xi (dy,x)$. Now, changing the order of summation and denoting by 
\begin{equation*}
h_{n}(y)\allowbreak =\allowbreak \hat{\alpha}_{n}\sum_{j=0}^{\infty }\lambda
_{nj}\beta _{j}(y),
\end{equation*}%
we have 
\begin{equation*}
dF(x,y)=d\mu (x)d\nu (y)\sum_{n\geq 0}h_{n}(y)\alpha _{n}(x)/\hat{\alpha}%
_{n}.
\end{equation*}%
Further, utilizing (\ref{pr1}) and (\ref{pr3}) we have: 
\begin{gather*}
\zeta (dx|y)=(\sum_{n\geq 0}h_{n}(y)\alpha _{n}(x)/\hat{\alpha}_{n})d\mu (x),
\\
\int_{\limfunc{supp}v}h_{n}(y)v(dy)\allowbreak =\allowbreak \delta _{n,0}.
\end{gather*}%
with $\delta _{n,m}$ denoting traditionally Kronecker's delta. By
assumptions concerning polynomials $\left\{ \alpha _{n}\right\} $ and by
Theorem \ref{main} we see that for all $y\in \limfunc{supp}\nu $, we have: 
\begin{gather*}
h_{n}(y)\allowbreak =\allowbreak \int_{\limfunc{supp}\mu }\alpha
_{n}(x)d\zeta (x|y),\sum_{j\geq 0}\int_{\limfunc{supp}v}\left\vert
h_{j}(y)\right\vert ^{2}v(dy)/\hat{\alpha}_{j}<\infty , \\
\text{consequently }\sum_{j\geq 0}\left\vert h_{j}(y)\right\vert ^{2}/\hat{%
\alpha}_{j}<\infty ,
\end{gather*}%
$\func{mod}v$.

Now let us assume the converse statement, i.e. that we have the converging
to a nonnegative function in mean-square series (\ref{bv}) with polynomials $%
\left\{ \alpha _{n}\right\} $ and the measure $\mu ,$ as described in the
assumptions, together with the condition (\ref{h2}) satisfied for almost
every $y$ belonging to some closed set that we will denote by $\limfunc{supp}%
\nu $. By Theorem \ref{main}, we deduce that if the series (\ref{bv})
converges to a nonnegative function, then there exists a family $d\zeta
(x|y) $ of positive measures absolutely continuous with respect to $\mu $
such that $\forall n\geq 0:$%
\begin{equation*}
h_{n}(y)=\int_{\limfunc{supp}\mu }\alpha _{n}(x)\zeta (dx|y).
\end{equation*}%
Moreover, we have 
\begin{equation*}
\zeta (dx|y)=g(x,y)d\mu (x).
\end{equation*}%
Now, if there exists a probability measure $\nu $ such that $\forall n\geq
0:\int_{\limfunc{supp}\nu }h_{n}(y)v(dy)\allowbreak =\allowbreak \delta
_{n,0}$, then we 
\begin{equation*}
\int_{\limfunc{supp}\nu }g(x,y)d\nu (y)\allowbreak =\allowbreak 1,
\end{equation*}%
hence 
\begin{equation*}
\int_{\limfunc{supp}\nu }\zeta (dx|y)d\nu (y)\allowbreak =\allowbreak d\mu
(x),
\end{equation*}%
as claimed.
\end{proof}

The rest of this section will be dedicated to the so-called Lancaster
expansions. In particular, we will be able to give now necessary and
sufficient conditions for these types of expansions. Let us recall that
Lancaster, in the series of papers \cite{Lancaster58}, \cite{Lancaster63(1)}%
, \cite{Lancaster63(2)}, \cite{Lancaster75}, considered and developed the
following question: given a bivariate distribution say $dF(x,y)$, its two
marginal distributions say $d\mu (x)$ and $d\nu (y)$ and the two sets of
polynomials, when is it possible to find the set of numbers $\left\{
c_{n}\right\} $ such that 
\begin{equation}
dF(x,y)=d\mu (x)d\nu (y)\sum_{n\geq 0}c_{n}\alpha _{n}(x)\beta _{n}(y)
\label{La}
\end{equation}%
almost everywhere in $\limfunc{supp}(\mu )\times \limfunc{supp}(\nu )$ with
respect to the product measure. In fact, Lancaster in his papers and also
his followers in their papers confined the problem to such bivariate
distributions $dF$ satisfying condition (\ref{L2}).

\begin{definition}
A class of bivariate distributions with margins identifiable by moments,
satisfying (\ref{L2}) and having expansion (\ref{La}) will be called
Lancaster class (of bivariate distributions), briefly (LC distributions).
\end{definition}

\begin{remark}
\label{PC}Notice, that if $F$ is of LC distribution, then we have:%
\begin{eqnarray*}
\int_{\limfunc{supp}\mu }\alpha _{n}(x)dF(dx,y) &=&c_{n}\beta _{n}(y)d\nu
(y), \\
\int_{\limfunc{supp}\nu }\beta _{n}(x)dF(x,dy) &=&c_{n}\alpha _{n}(x)d\mu
(x).
\end{eqnarray*}%
In other words, in terms used in probability, we can easily deduce that $%
\forall n\geq 1$ the conditional moments, i.e.: 
\begin{eqnarray*}
E(X^{n}|Y) &=&p_{n}(Y), \\
E(Y^{n}|X) &=&q_{n}(X),
\end{eqnarray*}%
respectively $\func{mod}(\nu )$ and $\func{mod}(\mu )$, where $p_{n}$ and $%
q_{n}$, are some polynomials of the full order\footnote{$p_{n}(x)$ is of
full order $n$ iff coefficient by $x^{n}$ is nonzero} $n.$
\end{remark}

\begin{definition}
Class of bivariate distributions with margins identifiable by moments,
having the property that all its conditional moments of the order, say, $n$
are polynomials of the full order $n$ will be called polynomial class (of
distributions) briefly PC distributions.
\end{definition}

As a corollary we have the following characterization of the Lancaster class
of distributions.

\begin{theorem}
Let us consider a bivariate distribution $F$ satisfying (\ref{L2}) with
margins identifiable by moments. Then $F$ is an LC distribution iff it is a
PC distribution.
\end{theorem}

\begin{proof}
The fact that every distribution of the Lancaster class belongs also to the
PC class was noted in Remark \ref{PC}. So now, let us assume that $F$
belongs to the PC class. By Theorem \ref{EXPA} we know that it can be
expanded in the series (\ref{bv}). Now we see that $\forall n\geq
1:h_{n}(y)\allowbreak =\int \alpha _{n}(x)\zeta (dx,y)\allowbreak
=\allowbreak E((\alpha _{n}(X)|Y=y).$ But, by our assumption, $h_{n}(y)$ has
to be a polynomial of the full order $n$ i.e. 
\begin{equation*}
h_{n}(y)=\sum_{j=0}^{n}\gamma _{n,j}\beta _{j}(y),
\end{equation*}%
where $\left\{ \beta _{n}\right\} $ are the polynomials orthogonal with
respect to the marginal measure $\nu .$ Hence we musta have $\gamma _{n,j}=0$
for $j>n$. Now changing the order of summation in (\ref{bv}) we get:%
\begin{equation*}
F(dx,dy)=d\mu (x)d\nu (y)\sum_{j=0}^{\infty }\beta _{j}(y)\sum_{n\geq
j}\gamma _{n,j}\alpha _{n}(x)/\hat{\alpha}_{n}.
\end{equation*}%
But by our assumption $E((\beta _{j}(Y)|X=x)\allowbreak $ is a polynomial of
the full order $j.$ So we have:%
\begin{equation*}
\int_{\limfunc{supp}\nu }\beta _{j}(y)F(dx,dy)=d\mu (x)\sum_{n\geq j}\gamma
_{n,j}\alpha _{n}(x)/\hat{\alpha}_{n}.
\end{equation*}%
Now, by the uniqueness of expansion, we deduce that $\forall n>j:\gamma
_{n,j}\allowbreak =\allowbreak 0.$
\end{proof}

Szab\l owski in the series of papers \cite{SzablPoly}, \cite{SzabMark}, \cite%
{SzabStac} considered Markov stochastic processes having two-dimensional
finite distributions belonging to PC class. Hence, now, in light of the
above-mentioned theorem, there is a possibility of expanding in the
Lancaster-like series the transition functions of such Markov processes.

Let us apply Theorem \ref{EXPA} to the analysis of the LC distributions or
more precisely to the analysis when the series 
\begin{equation}
\sum_{n\geq 0}c_{n}\alpha _{n}(x)\beta _{n}(y)  \label{cs}
\end{equation}%
converges to a nonnegative function of $\left( x,y\right) ,$ where
polynomials $\left\{ \alpha _{n}\right\} $ and $\left\{ \beta _{n}\right\} $
are defined as in the introduction to Section \ref{prob}. To simplify the
formulation of the theorem and the applications following it, let us assume
additionally that both families of polynomials $\left\{ \alpha _{n}\right\} $
and $\left\{ \beta _{n}\right\} $ are orthonormal with respect to the
measures $\mu $ and $v$ respectively.

Let us also denote by $C(\alpha ,\beta )$ set of all sequences $\left\{
c_{n}\right\} $ for which the sum (\ref{cs}) exists and is positive. A.E.
Koudu in \cite{Koudu96} showed that this set is convex (which is trivial,
see, e.g. Appendix) and moreover, compact with respect to the weak topology.
Hence the Choquet's theorem about extreme points can be applied. As a
corollary of the Theorem \ref{EXPA} we have the following result:

\begin{theorem}
\label{Lancaster}Let the numbers $\left\{ a_{n,j}\right\} $ and $\left\{
b_{n,j}\right\} $ be defined by the polynomials $\alpha _{n}$ and $\beta
_{n} $ in the following way:%
\begin{equation}
\alpha _{n}(x)=\sum_{j=0}^{n}a_{n,j}x^{j},~~\beta
_{n}(y)=\sum_{j=0}^{n}b_{n,j}y^{j}.  \label{anj}
\end{equation}%
The series (\ref{cs}) converges in the mean-squares to a positive function
if and only if the following system of recurrent equations 
\begin{eqnarray}
\sum_{j=0}^{n}a_{n,j}m_{j}^{(a)}(y) &=&c_{n}\beta _{n}(y),  \label{1st} \\
\sum_{j=0}^{n}b_{n,j}m_{j}^{(b)}(x) &=&c_{n}\alpha _{n}(x),  \label{2d}
\end{eqnarray}%
for $n\geq 0$, is satisfied for almost all $y\in \limfunc{supp}v$ and $x\in 
\limfunc{supp}\mu $, by the two polynomial moment sequences $\left\{
m_{n}^{(a)}(y)\right\} $ and $\left\{ m_{n}^{(b)}(x)\right\} $, that are
defined by some measures, that are absolutely continuous with respect to the
measures $\mu $ and $v$ respectively.
\end{theorem}

\begin{proof}
Firstly, under our assumptions we have $\hat{\alpha}_{n}\allowbreak
=\allowbreak \hat{\beta}_{n}\allowbreak =\allowbreak 1,$ hence following the
previous theorem we deduce that the series (\ref{cs}) converges in mean
squares to some positive function iff 
\begin{equation*}
E\alpha _{n}(x)\allowbreak =\allowbreak c_{n}\beta _{n}(y),
\end{equation*}%
where the expectation is taken with respect to some absolutely continuous
measure that is additionally parametrized by the parameter $y,$ that belongs
to the $\limfunc{supp}v.$ From this remark follows directly the first of the
above-mentioned equation. By a similar argument we deduce that the second
equation holds.
\end{proof}

\begin{remark}
\label{Momseq}Notice that if for $c_{n}\allowbreak =\allowbreak \rho ^{n},$
for $\left\vert \rho \right\vert <1,$ the series (\ref{cs}) is convergent to
a positive bivariate probability density, then so is the series 
\begin{equation*}
\sum_{j\geq 0}(\int_{[-1,1]}\rho ^{j}d\gamma (\rho ))\alpha _{j}(x)\beta
_{j}(y),
\end{equation*}%
for any probability distribution $\gamma $ such that $\gamma (\left\{
-1\right\} \cup \left\{ 1\right\} )\allowbreak =\allowbreak 0.$
\end{remark}

\begin{remark}
In fact, the equations (\ref{1st}) and (\ref{2d}) should be written in the
following, less legible but more precise, recursive way:

\begin{eqnarray}
m_{n}^{(a)}(y) &=&c_{n}\frac{b_{n,n}}{a_{n,n}}y^{n}+%
\sum_{j=0}^{n-1}(c_{n}b_{n,j}y^{j}-a_{n,j}m_{j}^{(a)}(y))/a_{n,n},
\label{1stb} \\
m_{n}^{(b)}(y) &=&c_{n}\frac{a_{n,n}}{b_{n,n}}y^{n}+%
\sum_{j=0}^{n-1}(c_{n}a_{n,j}y^{j}-b_{n,j}m_{j}^{(b)}(y))/b_{n,n}.
\label{2db}
\end{eqnarray}

for $n\geq 0$ with $m_{0}(x)\allowbreak =\allowbreak m_{0}(y)\allowbreak
=\allowbreak 1$.
\end{remark}

\begin{corollary}
\label{cnn}Let coefficients $\left\{ a_{n,n}\right\} $ and $\left\{
b_{n,n}\right\} $ be defined by (\ref{anj}). If the series (\ref{cs})
converges to a positive sum, then

a) $\sum_{n\geq 0}c_{n}^{2}<\infty ,$

further additionally:

b) if $\left\{ 0\right\} \in \limfunc{supp}\mu $ then, we have: $\infty
>\sum_{n\geq 0}c_{n}a_{n,0}b_{n,0}\geq 0,$

c) if $\limfunc{supp}\mu $ is unbounded, then the sequence $\left\{ c_{n}%
\frac{a_{n,n}}{b_{n,n}}\right\} $, is a moment sequence, if additionally $%
\limfunc{supp}v$ is also unbounded, then $\left\{ c_{n}^{2}\right\} $ must
be a moment sequence. 

d) if measures $\mu $ and $v$ are the same and have unbounded supports, then 
$\left\{ c_{n}\right\} $ must be a moment sequence.
\end{corollary}

\begin{proof}
Part a) follows the fact that the series (\ref{Ors}) converges in
mean-squares and that $\hat{p}_{n}\allowbreak =\allowbreak 1$ since we
consider only orthonormal polynomials. b) is obvious. c) Firstly, notice
that in all cases from the system of equations (\ref{1stb}) and (\ref{2db})
it follows that the leading coefficient in $m_{n}^{(a)}(y)$ and $%
m_{n}^{(b)}(y)$ must be respectively $c_{n}b_{n,n}/a_{n,n}$ and $%
c_{n}a_{n,n}/b_{n,n}.$ Now, if, say $\limfunc{supp}\mu $ is unbounded, then
from the fact that $\left\{ m_{n}^{(b)}(y)\right\} $ must be a moment
sequence, then so must be the sequence $\left\{
y^{n}c_{n}a_{n,n}/b_{n,n}\right\} .$ From this fact the first assertion
follows immediately. Now, if both sequences $\left\{ c_{n}\frac{a_{n,n}}{%
b_{n,n}}\right\} ,$ $\left\{ c_{n}\frac{b_{n,n}}{a_{n,n}}\right\} $ are the
moment ones, then is their their product (see the Appendix, below). Part d)
follows directly from c).
\end{proof}

\begin{remark}
The assertion d) of the above-mentioned corollary repeats in fact the result
of Tyan et al. presented in \cite{Ty76}.
\end{remark}

\begin{remark}
Theorem \ref{Lancaster}, at least theoretically, closes the problem of
finding conditions for the convergence to a positive bivariate function of
the infinite series (\ref{cs}). Namely, having two sequences of moments
(what is important given by the recursive formula) one can find their two
Laplace transforms and invert them obtaining two conditional measures $\chi
(.|x)$ and $\zeta (.|y)$ that are also defined by the conditions 
\begin{eqnarray*}
\int \beta _{n}(y)d\chi (y|x) &=&c_{n}\alpha _{n}(x), \\
\int \alpha _{n}(x)d\zeta (x|y) &=&c_{n}\beta _{n}(y).
\end{eqnarray*}%
The procedure to get these inverses is very difficult and long. On the way,
the procedure utilizes Nevalinna's theory as described say in \cite%
{Akhizer65}. Now, the question of summing the series (\ref{cs}) is solved by
the formula (\ref{suma}).
\end{remark}

\begin{example}
We will now present an example in which we show how having a given family of
orthogonal polynomials, and a moments sequence $\left\{ c_{n}\right\} ,$ one
finds a sequence of moments $\left\{ m_{n}(y)\right\} $. Then, having this
sequence one finds a sequence of orthogonal polynomials parametrized by $y.$
Then, by different means, including the analysis of the three-term
recurrence of this sequence, one finds the properties of the measure having
moments sequence $\left\{ m_{n}(y)\right\} $ and thus conclude, basing on
Theorem \ref{Lancaster}, that the series (\ref{cs}) with $\alpha
_{n}(x)\allowbreak =\allowbreak \beta _{n}(x)$ converges to a positive sum.
The way is long and it seems that each case would be enough for an article.
To shorten the conclusions and description, the example will concern Hermite
polynomials that are well known and the sequence $c_{n}\allowbreak
=\allowbreak \rho ^{n}$ for some $\left\vert \rho \right\vert <1,$  just to
illustrate the process. Note, that such sequence $\left\{ c_{n}\right\} $ is
a moment sequence. On the way, we will make use of the well-known properties
of these polynomials. Besides in this case it is easy just to guess the
conditioning measure $d\chi (.|y).$

Let us recall that the so-called (probabilistic) Hermite polynomials are
defined by the following three-term recurrence 
\begin{equation*}
H_{n+1}(x)=xH_{n}(x)-nH_{n-1}(x),
\end{equation*}%
with $H_{0}(x)\allowbreak =\allowbreak 1,$ and $H_{-1}(x)\allowbreak
=\allowbreak 0.$ It is known, that we have%
\begin{equation*}
\frac{1}{2\pi }\int_{-\infty }^{\infty }H_{n}(x)H_{m}(x)\exp
(-x^{2}/2)dx=\delta _{nm}n!.
\end{equation*}

Moreover, for all complex $x,y$ and $a$ the following expansions are true:%
\begin{eqnarray}
H_{n}(x) &=&n!\sum_{m=0}^{\left\lfloor n/2\right\rfloor }\frac{%
(-1)^{m}x^{n-2m}}{2^{m}m!(n-2m)!},  \label{wspH} \\
H_{n}(ax+\sqrt{1-a^{2}}y) &=&\sum_{m=0}^{n}\binom{n}{m}%
a^{m}(1-a^{2})^{(n-m)/2}H_{m}(x)H_{n-m}(y).  \label{Hsum}
\end{eqnarray}%
Notice that 
\begin{equation*}
H_{n}(0)\allowbreak =\allowbreak \left\{ 
\begin{array}{ccc}
0 & if & n\text{ is odd} \\ 
\frac{(-1)^{k}(2k)!}{2^{k}k!} & if & n=2k.%
\end{array}%
\right. 
\end{equation*}%
and also that the orthonormal version of Hermite polynomials is equal to $%
H_{n}(x)/\sqrt{n!}$, so for even $n$ we have 
\begin{equation*}
\left\vert H_{2k}(0)\right\vert ^{2}/(2k)!\allowbreak =\allowbreak \frac{%
(2k)!}{2^{2k}(k!)^{2}}\allowbreak =\frac{1}{2^{2k}}\binom{2k}{k}=\frac{1}{%
\sqrt{\pi k}}.
\end{equation*}%
We used here the well-known approximation $\binom{2k}{k}\frac{1}{2^{2k}}%
\allowbreak \cong \allowbreak \frac{1}{\sqrt{\pi k}}$. Thus, applying
Corollary \ref{cnn}, we see that every applicable sequence $\left\{
c_{n}\right\} $ must satisfy the following conditions: 
\begin{equation*}
\sum_{n\geq 0}c_{n}^{2}<\infty ,~\infty >\sum_{n\geq 0}c_{n}/\sqrt{n+1}\geq
0,
\end{equation*}%
$\left\{ c_{n}\right\} $ is a moment sequence.

Now, let us recall Remark \ref{Momseq} and examine the case $%
c_{n}\allowbreak =\allowbreak \rho ^{n}$ for some $\left\vert \rho
\right\vert <1.$ Notice, that this sequence satisfies the above-mentioned
conditions, since $\sum_{n\geq 0}\rho ^{n}/\sqrt{n+1}\allowbreak
=\allowbreak Li(\frac{1}{2};\rho )/\rho \geq 0,$ $\rho \in (-1,1)$ , where $%
Li(s;\rho )$ is the so-called polylogarithm function of order $s.$

We guess, that the measure $v(dx|y)$ has the density 
\begin{equation*}
g(x|y,\rho )=\frac{1}{\sqrt{2\pi (1-\rho ^{2})}}\exp (-\frac{(x-\rho y)^{2}}{%
2(1-\rho ^{2})}).
\end{equation*}

From this, we immediately deduce, that 
\begin{equation*}
\int_{-\infty }^{\infty }H_{n}(\frac{x-\rho y}{\sqrt{1-\rho ^{2}}}%
)g(x|y,\rho )dx=0,
\end{equation*}%
for $n\geq 1.$ Let us denote: 
\begin{equation*}
m_{n}(y)=\int_{-\infty }^{\infty }x^{n}g(x|y,\rho )dx.
\end{equation*}%
Now we change variables to $z\allowbreak =\allowbreak \frac{x-\rho y}{\sqrt{%
1-\rho ^{2}}}$ getting $x\allowbreak =\allowbreak \rho y+\sqrt{1-\rho ^{2}}%
z. $ Now, applying (\ref{Hsum}) we get:%
\begin{equation}
m_{n}(y)=\sum_{j=0}^{\left\lfloor n/2\right\rfloor }\binom{n}{2j}(1-\rho
^{2})^{j}(2j-1)!!\rho ^{n-2j}y^{n-2j},  \label{mn}
\end{equation}%
since for $k\allowbreak =\allowbreak 0,1,\ldots $ : 
\begin{equation*}
\int_{-\infty }^{\infty }z^{n}\exp (-\frac{z^{2}}{2})dz\allowbreak
=\allowbreak \left\{ 
\begin{array}{ccc}
0 & if & n=2k+1 \\ 
(2k-1)!! & if & n=2k%
\end{array}%
.\right.
\end{equation*}

Moreover, we get 
\begin{equation*}
\int_{-\infty }^{\infty }H_{n}(x)g(x|y,\rho )dx\allowbreak =\allowbreak
\int_{-\infty }^{\infty }H_{n}(\rho y+z\sqrt{1-\rho ^{2}})\exp (-\frac{z^{2}%
}{2})dz=\rho ^{n}H_{n}(y),
\end{equation*}%
after changing variables, as above, and applying (\ref{Hsum}). Thus, there
is no need to check that indeed the sequence (\ref{mn}) is the solution of
the system of equations (\ref{1stb}) with $\alpha _{n}(x)\allowbreak
=\allowbreak \beta _{n}(x)\allowbreak =\allowbreak H_{n}(x)$ for all $n\geq 0
$ and $c_{n}\allowbreak =\allowbreak \rho ^{n}.$

By the way, we have also:%
\begin{equation*}
\sum_{n\geq 0}m_{n}(y)\frac{t^{n}}{n!}=\exp (-t^{2}(1-\rho ^{2})/2+yt).
\end{equation*}

Returning to Remark \ref{Momseq} we see that the series 
\begin{equation*}
\sum_{n\geq 0}\frac{c_{n}}{n!}H_{n}(x)H_{n}(y),
\end{equation*}%
converges to a positive density iff $c_{n}\allowbreak =\allowbreak
\int_{-1}^{1}\rho ^{n}d\gamma (\rho ),$ $n\allowbreak =\allowbreak
0,1,\ldots ,$ where $\gamma $ is some probability distribution on $[-1,1],$
such that $\gamma (\left\{ -1\right\} \cup \left\{ 1\right\} )\allowbreak
=\allowbreak 0$. This fact was already noticed by Sarmanov and Bratoeva in
1967 in \cite{Sarman67}. Later this result was generalized by Griffith in 
\cite{Griff69} and Koudu in \cite{Koudu98} with Hermite polynomials replaced
by polynomials orthogonalizing gamma distribution (Griffith) and Poisson and
negative binomial (Koudu). The Koudu's results were later applied to
parameter testing of the chosen Lancaster bivariate distributions by Chen in 
\cite{Chen20}.
\end{example}

\appendix\textbf{Appendix. }{}\emph{A few facts of the moment problem and
the orthogonal polynomials. }

Let be $\alpha $ be a signed, finite measure on the real line. Then the
sequence of reals $\left\{ m_{n}\right\} _{n}$ defined by: $m_{n}=\int
x^{n}d\alpha (x)$, is called a \emph{sequence of moments or the moment
sequence (sm) }of the measure $\alpha \text{.}\ $Below we have the
surprising general result by Boas \cite{Boas39}.

\begin{theorem}
Any sequence $\left\{ m_{n}\right\} _{n\geq 0}$ of real numbers can be
represented in the form: $m_{n}=\int x^{n}d\alpha (x)\ $,$\ $ where $\int
|d\alpha (x)|<\infty \ $.
\end{theorem}

In other words any sequence of numbers is a moment sequence of some signed
measure.

We will be interested in sequences of moments of positive measures.

It turns out (see, e.g. \cite{Sim98}), that the sequence $\left\{
m_{n}\right\} $ is a moment sequence of some nonnegative measure $\alpha $,
iff it satisfies the following condition for $\forall n\geq 0:$ 
\begin{equation}
d_{n}=\det [m_{i+j}]_{0\leq i,j\leq n}\geq 0.  \label{han}
\end{equation}%
$\ $The sequence $\left\{ d_{n}\right\} $ related to the sequence $\left\{
m_{n}\right\} $ and defined by ( \ref{han}), is called the Hankel transform
of the sequence $\left\{ m_{n}\right\} .$ $\ $It is also known (see, e.g. 
\cite{Sim98}) that if the sharp inequality in (\ref{han}) holds for all $%
n\geq 1$ then the support of $\alpha $ is \emph{infinite}. If additionally $%
\forall n\geq 0:$ 
\begin{equation*}
\det [m_{1+i+j}]_{0\leq i,j\leq n}\geq 0,
\end{equation*}%
then $\limfunc{supp}\alpha \subset \lbrack 0,\infty ).$

Sequences $\left\{ m_{n}\right\} $ that are the moment sequences of some
nonnegative measures will me called p(ositive)m(oment) sequences i.e. pm
sequences.

Let us mention the two simple necessary conditions for a sequence of reals
to be a pm sequence.

\begin{condition}[Necessary]
Let $\left\{ d_{n}\right\} _{n\geq 0}$ a pm sequence. Then a) $d_{2k}\geq 0,$
$k\allowbreak =\allowbreak 0,1,\ldots $ . b) $\left\vert d_{k}\right\vert
^{2}\leq d_{2k}d_{0},$ c) sequence $\left\{ d_{2k}^{1/(2k)}\right\} $ is
non-decreasing.
\end{condition}

\begin{proof}
Assertion a) is obvious, b) follows directly Cauchy-Schwarz inequality while
c) follows Jensen's inequality.
\end{proof}

In the sequel we will assume that the measure $\alpha $ is a probability
measure i.e. $\int d\alpha (x)=1$ which results in the fact that all pm
sequences will have $m_{0}=1.$

The generating function $\varphi $ of the pm sequence is defined by the
following formula:%
\begin{equation*}
\varphi (t,\alpha )\allowbreak =\allowbreak \sum_{n\geq
0}t^{n}m_{n}/n!\allowbreak =\allowbreak \sum_{n\geq 0}\int_{\limfunc{supp}%
\mu }\frac{(xt)^{n}}{n!}d\alpha (x)\allowbreak =\allowbreak \int_{\limfunc{%
supp}\mu }\exp (tx)d\alpha (x).
\end{equation*}

Hence, if the moment problem is determinate (i.e. the measure $\alpha $ is
identified by its sequence of moments) if this Laplace transform exists for
even small neighborhood of zero.

For the aims of this paper, this criterion of determinacy is more important.
However, for the sake of completeness, let us remark that there exists,
however, another criterion given by Carleman (see, e.g. \cite{Akhizer65} or 
\cite{Chih79}) where the determinacy follows the properties of the moment
sequence itself. Namely, Carleman's criterion reads: If only%
\begin{equation*}
\sum_{n\geq 1}m_{2n}^{-1/(2n)}=\infty ,
\end{equation*}%
then the sequence of moments $\left\{ m_{n}\right\} $ defines uniquely the
measure that created this sequence.

It is known (see, e.g. \cite{Akhizer65} or \cite{Chih79}), that the sequence
of polynomials orthogonal with respect to the measure that produced a given
moment sequence $\left\{ m_{n}\right\} $ is given by the following sequence:%
\begin{equation*}
p_{n}(x)=\det \left[ 
\begin{array}{cccc}
m_{0} & m_{1} & \ldots & m_{n} \\ 
m_{1} & m_{2} & \ldots & m_{n+1} \\ 
\ldots & \ldots & \ldots & \ldots \\ 
0 & x & \ldots & x^{n}%
\end{array}%
\right] ,
\end{equation*}%
$n\allowbreak =\allowbreak 1,2,\ldots $ $.$

Further, it is also known (see, e.g. \cite{Akhizer65} or \cite{Chih79}) that
for every orthogonal polynomial sequence $\left\{ p_{n}\right\} $ one can
define three sequences of numbers $\left\{ A_{n}\right\} ,$ $\left\{
B_{n}\right\} ,$ $\left\{ C_{n}\right\} $, such that for every $n\geq 0:$%
\begin{equation*}
p_{n+1}(x)\allowbreak =\allowbreak (A_{n}x+B_{n})p_{n}(x)-C_{n}p_{n-1}(x),
\end{equation*}%
and for $n\geq 1:C_{n}A_{n}A_{n-1}>0$, provided the support of the measure
making these polynomials orthogonal is infinite. These real sequences are
defined by numbers $\left\{ d_{n}\right\} $ given by (\ref{han}). For
details see again \cite{Akhizer65} or \cite{Chih79}.

We have also the following simple observations concerning the properties of
pm sequences.

\begin{proposition}
Let $\left\{ a_{n}\right\} _{n\geq 0}$ and $\left\{ b_{n}\right\} _{n\geq 0}$
be two pm sequences. Then, so are the following sequences:

1. $\left\{ pa_{n}+b_{n}(1-p)\right\} _{n\geq 0}\text{, }$for $p\in \lbrack
0,1]$, $\left\{ \sum_{i=0}^{n}(\pm 1)^{i}\binom{n}{i}\alpha ^{i}\beta
^{n-i}a_{i}b_{n-i}\right\} _{n\geq 0}$, for $\alpha ,\beta \in \mathbb{R}$ , 
$\left\{ a_{n}b_{n}\right\} _{n\geq 0},$

2.$\ $ $\left\{ a_{kn}\right\} _{n\geq 0}\text{, }k\in \mathbb{N}$,$\ $ $%
c_{n}=\left\{ 
\begin{array}{ccc}
a_{2k} & if & n=2k \\ 
0 & if & n=2k+1%
\end{array}%
\right. \ \text{,}\ $ $k=0,1,\ldots $ . If a pm sequence $\left\{
a_{n}\right\} _{n\geq 0}$ is nonnegative then also pm is the following
sequence: $b_{n}^{\ ^{\prime }}=\left\{ 
\begin{array}{ccc}
0 & if & n=2k+1 \\ 
a_{k} & if & n=2k%
\end{array}%
\right. \ \text{.}\ $

3. The following sequences : $\forall a\in \mathbb{R}:$%
\begin{equation*}
\left\{ a^{n}\right\} _{n\geq 0}\ \text{,}\ \left\{ 
\begin{array}{ccc}
1 & if & n=0 \\ 
0 & if & n\text{ is odd} \\ 
(2k-1)!! & if & n=2k%
\end{array}%
\right.
\end{equation*}%
$,\allowbreak k=1,2,\ldots \ \text{,}\ $ Catalan numbers i.e. $\left\{ 
\binom{2n}{n}/(n+1)\right\} _{n\geq 0}\ \text{,}\ $ $\left\{ n!\right\}
_{n\geq 0}\ \text{,}\ \forall k>-1:\{1/(n+1)^{k+1}\}_{n\geq 0},$~ $\left\{
F_{n+1}\right\} _{n\geq 0}\ \text{,}\ $ $\left\{ F_{n+1}/(n+1)\right\}
_{n\geq 0}\ \text{,}\ $ $\left\{ F_{2n+2}/(n+1)\right\} _{n\geq 0}\ \text{,}%
\ $ $\left\{ (F_{2n+1}-1)/(n+1)\right\} _{n\geq 0}$ where $F_{n}$ denotes $%
n- $th Fibonacci number, are pm sequences.
\end{proposition}

\begin{proof}
1. The arguments are probabilistic. Let $X$ and $Y$ be two independent
random variables having respectively moments $\left\{ a_{n}\right\} $ and $%
\left\{ b_{n}\right\} \text{.}\ $ Then $\left\{ a_{n}b_{n}\right\} $ is the
moment sequence of $XY\ \text{,}\ $ $\left\{ \sum_{i=0}^{n}(\pm 1)^{i}\binom{%
n}{i}\alpha ^{i}\beta ^{n-i}a_{i}b_{n-i}\right\} _{n\geq 0}$ is the moment
sequence of $(\beta Y\pm \alpha X).$ Let $Z$ have the so-called mixture
distribution of the distributions of random variables $X$ and $Y$ having
moment sequences, respectively $\left\{ a_{n}\right\} $ and $\left\{
b_{n}\right\} $ . Then $\left\{ pa_{n}+b_{n}(1-p)\right\} _{n\geq 0}$ is the
moment sequence of $Z\ \text{.}\ $ $\ $

2. $\left\{ a_{kn}\right\} _{n\geq 0}$ is the moment sequence of $X^{k}$. To
get remaining statements of this we consider special mixtures of the
independent copies of $X$ and $-X$ to get first statement and $\sqrt{X}$ and 
$-\sqrt{X}$ to get the second.

2. $\left\{ a^{n}\right\} _{n\geq 0}$ is the moment sequence of the
one-point distribution concentrated at $a\ $,$\ $ 
\begin{equation*}
\left\{ 
\begin{array}{ccc}
1 & if & n=0 \\ 
0 & if & n=2k-1 \\ 
(2k-1)!! & if & n=2k%
\end{array}%
\right. ,
\end{equation*}%
$k=1,2,\ldots $ is the moment sequence of Normal $N(0,1)$ distribution,
Catalan numbers are moments of distribution with the density $\frac{1}{2\pi }%
\sqrt{\frac{4-x}{x}}\ \text{,}\ $ $x\in (0,4)\ \text{.}\ $ $\left\{
n!\right\} _{n\geq 0}$ are moments of distribution with the density $\exp
(-x),x\geq 0,$~ $\{1/(n+1)^{k+1}\}_{n\geq 0}$ are the moment sequence of
distributions with the densities $(-\log (x))^{k}/\Gamma (k+1),$ $x\in (0,1)$%
, $k>-1$. For sequences composed of Fibonacci numbers, see \cite{Benn11}.
\end{proof}

Hence in particular the following families of polynomials are the pm
sequences for every $x\in \mathbb{R}$:

$\left\{ \sum_{k=0}^{n}\binom{n}{k}a_{k}(\pm 1)^{n-k}x^{n-k}\right\} _{n\geq
0},$ where $\left\{ a_{n}\right\} _{n\geq 0}$ is the pm sequence hence some
of the Appell polynomial sequences are pm.

\end{document}